\documentclass{aims}
\usepackage{amsmath}

  \usepackage{paralist}
\usepackage{mathtools}

\mathtoolsset{showonlyrefs}
  \usepackage{graphics} %% add this and next lines if pictures should be in esp format
  \usepackage{epsfig} %For pictures: screened artwork should be set up with an 85 or 100 line screen
\usepackage{graphicx}  \usepackage{epstopdf}%This is to transfer .eps figure to .pdf figure; please compile your paper using PDFLeTex or PDFTeXify.
 \usepackage[colorlinks=true]{hyperref}
\hypersetup{urlcolor=blue, citecolor=red}
\usepackage{url}
\def\currentvolume{X}
 \def\currentissue{X}
  \def\currentyear{200X}
   \def\currentmonth{XX}
    \def\ppages{X--XX}
     \def\DOI{10.3934/xx.xx.xx.xx}

\newtheorem{theorem}{Theorem}[section]

\newtheorem{proposition}{Proposition}

\theoremstyle{definition}
\newtheorem{definition}[theorem]{Definition}

\title[nonautonomous stability of GLMs] %Use the shortened version of the full title
      {Underlying one-step methods and nonautonomous stability of general linear methods}

\author[Andrew J. Steyer and Erik S. Van Vleck]{}


 \email{asteyer@sandia.gov}
 \email{erikvv@ku.edu}

\thanks{This research was supported in part by NSF grant DMS-1419047.}

\thanks{$^*$ Corresponding author: Andrew J. Steyer}

\begin{document}
\maketitle

% Enter the first author's name and address:
\centerline{\scshape Andrew J. Steyer$^*$}
\medskip
{\footnotesize
% please put the address of the first author
 \centerline{Sandia National Laboratories}
 \centerline{P.O. Box 5800, MS 1320}
  \centerline{Albuquerque, NM 87185-1320, USA }
 %  \centerline{Other lines}
   %\centerline{ Springfield, MO 65801-2604, USA}
} % Do not forget to end the {\footnotesize by the sign }

\medskip

\centerline{\scshape Erik S. Van Vleck}
\medskip
{\footnotesize
 % please put the address of the second  and third author
\centerline{Department of Mathematics - University of Kansas}
 \centerline{1460 Jayhawk Blvd}
 \centerline{ Lawrence, KS 66045-7594, USA}
}

\bigskip

% The name of the associate editor will be entered by an editorial staff
% "Communicated by the associate editor name" is not needed for special issue.
 \centerline{\it Dedicated to the memory of Timo Eirola.}

%The abstract of your paper
\begin{abstract}
We generalize the theory of underlying one-step methods to strictly stable general linear methods (GLMs) solving nonautonomous ordinary differential equations (ODEs) that satisfy a global Lipschitz condition.  We combine this theory with the Lyapunov and Sacker-Sell spectral stability theory for one-step methods developed in \cite{steyerthesis,SVV2016,SVV1} to analyze the stability of a strictly stable GLM solving a nonautonomous linear ODE.  These results are applied to develop a stability diagnostic for the solution of nonautonomous linear ODEs by strictly stable GLMs.
\end{abstract}

\section{Introduction}\label{sec:introduction}

'What do multistep methods approximate?' This question is one that beleaguers many researchers of multistep discretizations of ordinary differential equation (ODE) initial value problems (IVPs) due to the fact that the local truncation error of a $k$-step multistep method depends on the previous $k$ steps.  For time-independent (autonomous) ODEs, two classic papers, \cite{K1986} and \cite{EN1988}, provide an answer to this question by applying invariant manifold theory for maps to relate the numerical solution produced by a multistep method to the flow of the differential equation it is approximating.  The focus of this paper is to use the spirit and technique of \cite{K1986} and \cite{EN1988} together with invariant manifold theory for time-dependent (nonautonomous) difference equations to develop a stability theory for general linear methods (GLMs) solving nonautonomous linear ODEs.

Our contribution in this work is twofold.  We first apply invariant manifold theory for nonautonomous difference equations to prove Theorem \ref{thm:maintheorem1} which applies to strictly stable GLMs solving a nonautonomous ODE that satisfies a global Lipschitz condition in its state variables.  This theorem states that for sufficiently small step-sizes there exists a time-independent linear change of coordinates and a unique continuous function whose graph defines a one-step method (called the underlying one-step method) with local truncation error the same order as the GLM.  The one-step method is a globally exponentially attractive, invariant manifold of the discrete-time system resulting from the time-independent change of variables.  Theorem \ref{thm:maintheorem1} generalizes the technique of characterizing the approximation properties of a GLM by its underlying one-step method to ODEs that are nonautonomous.

The second contribution of this paper is to use Theorem \ref{thm:maintheorem1} and the Lyapunov and Sacker-Sell spectral stability theory for one-step methods solving nonautonomous ODE IVPs developed in \cite{steyerthesis,SVV2016,SVV1} to prove Theorem \ref{thm:maintheorem2}.  Theorem \ref{thm:maintheorem2} states that for all sufficiently small step-sizes the numerical solution by a strictly stable GLM of a uniformly or non-uniformly exponentially stable nonautonomous linear ODE is exponentially stable.  While our analysis is unable to show uniform exponential stability even if the ODE is uniformly exponentially stable, Theorem \ref{thm:maintheorem2} still provides a way of analyzing the numerical stability of time-dependent linear ODEs that may fail to satisfy the hypotheses of AN- and B-stability theory (for an example of such an ODE see equation \eqref{eq:2dlin} below).  Subsequently, we apply Theorem \ref{thm:maintheorem2} to prove Proposition \ref{thm:nonlinapprox} showing that a strictly stable GLM approximating a uniformly exponentially stable trajectory of a nonlinear ODE will produce an (non-uniformly) exponentially stable numerical solution.  The theoretical results are used to develop a Lyapunov exponent based stability diagnostic to determine when a strictly stable GLM fails to produce a decaying numerical solution to a linear ODE whose Lyapunov or Sacker-Sell spectrum is bounded above by zero.

The use of invariant manifold theory to characterize the approximation properties of a multistep method by an associated one-step method was pioneered in \cite{K1986} and \cite{EN1988}.  The results of \cite{K1986} were extended to GLMs in \cite{Stoffer1993} using the invariant manifold theory for maps developed in \cite{NippStoffer1992}.  Understanding the properties of the underlying one-step method of a GLM is critical for understanding the evolution of the numerical solution (see for example the papers \cite{Hairer2008} and \cite{DHZ2013} as well as the book \cite{HLW}).  We extend existing theoretical techniques for invariant manifold reduction of GLMs for autonomous differential equations to nonautomous equations by employing the invariant manifold theory for nonautonomous differential and difference equations (see \cite{aulbach1998,AWP2002,ARS,AW2003}).  We remark here that we can prove the existence of an underlying one-step method for strictly stable GLMs solving nonautonomous ODEs by applying the standard technique of converting a nonautonomous ODE to an autonomous ODE in an extended phase space of one higher dimension (see Section 4.2 of \cite{steyerthesis} for the special case of strictly stable linear multistep methods).  However, we give an example below (Equation \eqref{eq:naode}) that shows that this type of reduction excludes nonautonomous ODEs satisfying the hypotheses of the theory developed in this paper.

The stability of the numerical solution of an ODE IVP by a multistep method is a challenging and important problem dating back at least to the investigations by Dahlquist in \cite{Dahlquist1956,Dahlquist1959,Dahlquist1963}.  For time-dependent trajectories the well-established stability theories (e.g. AN-stability, B-stability, algebraic stability) give conditions on a GLM so that, with no step-size restriction, it preserves the asymptotic decay of a trajectory that is uniformly decaying.  This restricts the analysis to implicit methods and there is no obvious analog of linear stability domains for time-dependent problems.  In this paper we exploit the Lyapunov and Sacker-Sell spectral stability theory for one-step methods developed in \cite{steyerthesis,SVV2016,SVV1} with theoretical results on underlying one-step methods on GLMs solving nonautonomous ODEs developed herein to characterize the stability of a strictly stable GLM solving a nonautonomous linear ODE whose Lyapunov or Sacker-Sell spectrum is bounded above by zero.

We now give an example of a nonautonomous ODE which, when viewed as an autonomous ODE in one higher dimension, does not satisfy the hypotheses of the theory for underlying one-step methods of GLMs solving autonomous ODEs.  Consider the scalar ODE 
\begin{equation}\label{eq:naode}
\dot{x}=ax +\tanh(t^2), \quad a \in \mathbb{R}, \quad t\in \mathbb{R}.
\end{equation}
The function $f(x,t):=ax + \tanh(t^2)$ is analytic, bounded in $t$ for each fixed $x$, and satisfies the global in space Lipschitz condition $|f(y,t)-f(x,t)| \leq |a||y-x|$ for all $x,y,t \in \mathbb{R}$.  The ODE \eqref{eq:naode} therefore satisfies the hypotheses of Theorem \ref{thm:maintheorem1}.  Using the standard substitution $\tau(t) =t$ and $\dot{\tau} = 1$ we can view the nonautonomous scalar ODE \eqref{eq:naode} as the following two dimensional autonomous ODE 
\begin{equation}\label{eq:naode2}
\left\{
\begin{array}{lcr}
\dot{x}=ax +\tanh(\tau^2) \\
\dot{\tau} = 1
\end{array}
\right.
\end{equation}
The right-hand side of the ODE \eqref{eq:naode2} is not Lipschitz since $\tanh(\tau^2)$ is not.  Thus the ODE \eqref{eq:naode2} fails to satisfy the hypotheses of the theories developed in \cite{K1986} and \cite{EN1988}. 

We motivate the development of our nonautonomous stability theory for GLMs with the following example.  Consider the ODE 
\begin{equation}\label{eq:2dlin}
\dot{x} = A(t)x, \quad A(t) = Q(t)B(t)Q(t)^T + \dot{Q}(t)Q(t)^T, \quad t > 0
\end{equation}
$$B(t) = \left[\begin{array}{cc}a_1\cos(t)+b_1 & \beta \\ 0 & a_2\cos(t)+b_2\end{array}\right], \quad Q(t) = \left[\begin{array}{cc}\cos(\omega(t)) & -\sin(\omega (t)) \\ \sin(\omega (t)) & \cos(\omega (t))\end{array} \right], $$
$$ b_2 < b_1 < 0, \quad a_1 ,a_2 > 0, \quad\omega \in C^2((0,\infty))$$
solved by the BDF2 method with constant step-size $h > 0$:
\begin{equation}\label{eq:bdf2}
x_{n+2}  - \frac{4}{3}x_{n+1} + \frac{1}{3}x_n= \frac{2}{3}hf(x_{n+2},t_{n+2}), \quad t_n :=nh
\end{equation}
The BDF2 method is a $3$-step and single-stage strictly stable GLM that is AN- and B-stable and has local truncation error of order $2$.  The ODE \eqref{eq:2dlin} does not satisfy the one-sided Lipschitz estimates of B-stability theory nor does it satisfy the hypotheses of AN-stability theory.  However, there exists $K > 0$ so that every solution $x(t)$ of \eqref{eq:2dlin} satisfies that $\|x(t)\| \leq K \|x(s)\| e^{b_1 (t-s)}$ where $t \geq s$ and $\| \cdot\|$ is some norm on $\mathbb{R}^{2}$.  We show that given any step-size $h > 0$, there is a suitable choice of parameters such that BDF2 can produce an exponentially growing solution to \eqref{eq:2dlin}.  Let $I_2$ denote the $2 \times 2$ identity matrix, $0_2$ denote the $2 \times 2$ matrix of zeros, $\{x_n\}_{n=0}^{\infty}$ denote the numerical solution with $0 \neq x_0 \in \mathbb{R}^{2}$, and let $X_n = (x_n^T,x_{n+1}^T)^T$.  Let $h > 0$ be such that $I_2-h A(t)$ is invertible for all $t \geq 0$ and suppose that $\alpha_i := a_i+b_i > 0$ for $i=1,2$ and $\omega(t)$ is such that  $\dot{\omega}(nh) = 2\pi/h$ for $n \geq 0$.  We then have that
$${X_{n+1} = \left[\begin{array}{cc}0_2 & I_2 \\ -\frac{1}{3}(I-hA(t_{n+2}))^{-1} & \frac{4}{3}(I-hA(t_{n+2}))^{-1}\end{array} \right] X_n \equiv C X_n}$$
where $C=\left[\begin{array}{cc}0_2 & I_2 \\ -\frac{1}{3}D^{-1} & \frac{4}{3}D^{-1}\end{array} \right]$ and $D = \left[\begin{array}{cc} 1-a_2 & 2\pi/h-\beta \\ -2\pi/h & 1-a_2\end{array}\right]$.  The matrix $C$ has an eigenvalue with modulus greater than $1$ for all sufficiently large $\beta$ and it follows that $\|X_n\| \rightarrow \infty$ exponentially fast as $n \rightarrow \infty$ for some initial value $X_0$.

Our main stability result (Theorem \ref{thm:maintheorem2}) follows from the Lyapunov and Sacker-Sell stability theory for one-step methods developed in \cite{steyerthesis,SVV2016,SVV1}.  These results use the local truncation error of a method to characterize its stability.  This confounds the concepts of accuracy and stability which have traditionally been considered separate topics in the analysis of initial value problem solvers.  The reason for blurring the separation between these concepts is based on classical ideas.  Consider a nonautonomous complex scalar test equation 
\begin{equation}\label{eq:testeq}
\dot{z} = \lambda(t)z, \quad \lambda:(t_0,\infty)\rightarrow \mathbb{C}.
\end{equation}
Suppose that we solve \eqref{eq:testeq} with a method $\mathcal{M}$ that has linear stability region $\mathcal{D}_{\mathcal{M}}$.  If the Sacker-Sell spectrum of \eqref{eq:testeq} lies to the left of zero, then zero is a uniformly exponentially stable equilibrium of \eqref{eq:testeq}.  Under this assumption it is possible that for some step-size $h > 0$, there exists a sequence $\{t_n\}_{n=0}^{\infty} \subset [t_0,\infty)$ where $t_n =t_0 + nh$ such that $h\lambda(t_n) \notin \mathcal{D}_M$ for all $n \geq 0$.  It is shown in \cite{steyerthesis,SVV1} that the coefficients $\lambda(t)$ of the test problems for \eqref{eq:2dlin} are approximately $b_i + a_i\cos(t)$, $i=1,2$.  The reason BDF2 failed to produce a decaying solution to \eqref{eq:2dlin} is that $h\lambda(t) \approx h(a_i\cos(t)+b_i)$ can cross the boundary of the stability region of BDF2 at the origin infinitely often and the step-size must be restricted to make sure that the average sign of $\lambda(t)$ is accurately approximated or equivalently that the stability spectra of the numerical method accurately approximate the stability spectra of the differential equation it is solving.  We show in Theorem \ref{thm:counterexample} in Section \ref{sec:mainresults2} that no Runge-Kutta or strictly stable linear multistep method can produce a decaying solution to every uniformly exponentially stable test equation \eqref{eq:testeq} without a step-size restriction.

The remainder of this paper is organized as follows. In Section 2 we introduce some definitions and background material on the approximation of Lyapunov and Sacker-Sell spectral intervals based on smooth $QR$ decompositions of fundamental matrix solutions and the associated nonautonomous stability theory for one-step methods.  In Section \ref{sec:mainresults1} we prove an existence theorem (Theorem \ref{thm:maintheorem1}) for underlying one-step methods of strictly stable GLMs.  In Section \ref{sec:mainresults2} we apply Theorem \ref{thm:maintheorem1} to prove Theorem \ref{thm:maintheorem2} which relates the stability of a strictly stable GLM solving a nonautonomous linear ODE to the Lyapunov and Sacker-Sell spectrum of its underlying one-step method.  In Section \ref{sec:experiments} we present the results of two experiments showing how the theory developed in Section \ref{sec:mainresults} can be used to develop a Lyapunov exponent based stability diagnostic for strictly stable GLMs solving linear ODEs.  The paper is concluded with some final remarks in Section \ref{sec:conclusion}.

\section{Preliminaries}\label{sec:preliminaries}

%In this section we introduce some notation and terminology, present some results from the theory of Lyapunov and Sacker-Sell spectra and their approximation using $QR$ based methods, and state the main results from the nonautonomous stability theory for one-step methods necessary for the development of the theory in this paper.  

For the remainder of this work we let $\|\cdot\|$ be a norm on $\mathbb{R}^{d}$ and use the same symbol for the induced matrix norm.  We may sometimes drop writing the explicit $t$ dependence of matrices and functions when their time dependence is clear from the context.  Whenever we use the word stability we are referring to time-dependent Lyapunov stability.  Consider the well-posed ODE IVP with sufficiently smooth $f$
\begin{equation}\label{eq:odeivp}
\left\{
\begin{array}{lcr}
\dot{x} = f(x,t) \\
x(t_0) = x_0
\end{array}\right.
\end{equation}
where $f:\mathbb{R}^{d} \times (\tau_0,\infty) \rightarrow \mathbb{R}^{d}$, $\tau_0 \geq -\infty$, and $t_0 > \tau_0$.  We consider numerical solutions of \eqref{eq:odeivp} by a fixed step-size, $k$-step, $r$-stage general linear method 
\begin{equation}\label{eq:glm}\left\{\begin{array}{lcr}
G_n = (U \otimes I_d)X_n +h(C \otimes I_d)F_n\\
X_{n+1} = (V \otimes I_d)X_n +h(D \otimes I_d)F_n
\end{array}\right.
\end{equation}
where $h > 0$ is the size of the time step (the step-size), $t_n = t_0 + nh$, $I_d$ is the $d\times d$ identity matrix, $U  \in \mathbb{R}^{r \times k}$, $V \in \mathbb{R}^{k \times k}$, $C \in \mathbb{R}^{r \times r}$, $D\in \mathbb{R}^{k \times r}$, $F_n = (f_{n,1},\hdots,f_{n,r})^T \in \mathbb{R}^{dr}$ where $f_{n,i} =f(g_{n,i},t_n+\xi_i h)$ for some real constants $\xi_i$ and $i=1,\hdots,r$, and $G_n = (g_{n,1}^T,\hdots,g_{n,r}^T)^T \in \mathbb{R}^{dr}$.  The symbol $\otimes$ denotes the Kronecker matrix product which defines an algebraic operation on matrices $A =(a_{i,j}) \in \mathbb{R}^{m \times n}$, $B \in \mathbb{R}^{p\times q}$ for positive integers $m,n,p,q$ by the rule
$$A \otimes B = \left[\begin{array}{ccccc} a_{1,1}B & \hdots & a_{1,n}B \\
\vdots & \ddots & \vdots \\ a_{m,1}B & \hdots & a_{m,n}B
 \end{array} \right]$$
An important property of Kronecker products that we use in Section \ref{sec:mainresults} is that if $A$ and $B$ are invertible, then $(A \otimes B)^{-1} = (A^{-1} \otimes B^{-1})$. A general linear method \eqref{eq:glm} is said to be strictly stable if $1$ is an eigenvalue of $V$ and all the other eigenvalues of $V$ have modulus strictly less than $1$.   We refer readers to \cite{Butcher1987} and \cite{Jackiewicz2009} for excellent overviews of the theory of general linear methods. % Strict stability is not a particularly stringent hypothesis to place on a GLM since all Runge-Kutta methods are strictly stable as are many popular linear multistep methods such as BDF1-3 and the Adams-Bashforth methods. 

To use a GLM \eqref{eq:glm} to approximate the solution of \eqref{eq:odeivp} we need a starting procedure $\mathcal{S}_h:\mathbb{R}^{d} \rightarrow \mathbb{R}^{kd}$ and a finishing procedure $\mathcal{F}_h:\mathbb{R}^{dk}\rightarrow \mathbb{R}^{d}$ such that $\mathcal{F}_h \circ \mathcal{S}_h = \text{id}$.  The starting procedure takes the initial condition $x_0$ into an initial value $X_0$ for the method \eqref{eq:glm} and a finishing procedure takes values $X_n$ produced by the method and turns them into an approximation to the solution of \eqref{eq:odeivp} at $t_n$. Let $F:\mathbb{R}^{dk} \times \mathbb{Z} \times (0,\infty) \rightarrow \mathbb{R}^{dk}$ denote the map defined by the output $X_n$ of the GLM \eqref{eq:glm} using step-size $h > 0$ such that $X_{n+1} = F(X_n,n,h)$.  A general linear method is said to have local truncation error of order $p$ relative to a starting procedure $\mathcal{S}_h$ if
$$F(\mathcal{S}_h(v(t)),n,h)-\mathcal{S}_h(v(t+h)) = \mathcal{O}(h^{p+1})$$
for every sufficiently smooth function $v(t)$, $n \geq 0$, and sufficiently small $h$.  A general linear method is said to have local truncation error of order $p$ if $p$ is the maximal positive integer such that the method has local truncation error of order $p$ relative to some starting procedure $\mathcal{S}_h$.  

If $f(x,t)$ is continuously differentiable on its domain, then associated to the solution $x(t;x_0,t_0)$ of \eqref{eq:odeivp} is the linear variational equation  
\begin{equation}\label{eq:lineq}
\dot{x}=Df(x(t;x_0,t_0),t)x \equiv A(t)x, \quad t > t_0, \quad D= \frac{\partial}{\partial x}
\end{equation}
where $A:(t_0,\infty)\rightarrow \mathbb{R}^{d}$.  The stability of the zero solution of \eqref{eq:lineq} in general does not depend on the time-dependent eigenvalues of $A(t)$ (see the example at the bottom of page 3 of \cite{Coppel} or the third example on page 24 of \cite{Kreiss1978}) which has motivated the development of several alternative stability spectra.  The two spectra we consider in this work are the Lyapunov spectrum, based on the theory of Lyapunov exponents originating in \cite{lyap}, and the Sacker-Sell spectrum, which was developed in \cite{SackerSell1978}.  

We say that the zero solution of a nonautonomous linear ODE of the form \eqref{eq:lineq} is uniformly exponentially stable (or simply that the nonautonomous linear ODE is uniformly exponentially stable) if for any fundamental matrix solution $X(t)$, there exists $K > 0$ and $\gamma > 0$ so that 
$$\|X(t)\| \leq K e^{-\gamma(t-s)}\|X(s)\|, \quad t \geq s \geq t_0.$$
Analogously, we say that the zero solution of \eqref{eq:lineq} is exponentially stable (or simply that \eqref{eq:lineq} is exponentially stable) if for any fundamental matrix solution $X(t)$, there exists $K > 0$ and $\gamma > 0$ so that 
$$\|X(t)\| \leq Ke^{-\gamma (t-t_0)}\|X(t_0)\|, \quad t \geq t_0.$$
A sufficient condition for uniform exponential stability is that the Sacker-Sell spectrum is bounded above by zero and a sufficient condition for exponential stability is that the Lyapunov spectrum is bounded above by zero. 

These linear stability concepts have natural extensions for solutions of nonlinear differential equations.  For $s \geq t_0$, let $x(t;u_{s},s)$ denote the solution of the differential equation $\dot{x} = f(x,t)$ from \eqref{eq:odeivp} with the initial condition $x(s;u_{s},s) = u_{s}$.

\begin{definition}
We say that the solution $x(t;x_0,t_0)$ of \eqref{eq:odeivp} is exponentially stable if there exists $K,\gamma,\delta > 0$ so that if $\|u_0-x_0\| < \delta$ and $t \geq t_0$, then $\|x(t;u_0,t_0)-x(t;x_0,t_0)\| \leq Ke^{-\gamma(t-t_0)}\|u_0-x_0\|$.  We say that $x(t;x_0,t_0)$ is uniformly exponentially stable if there exists $\delta, K , \gamma > 0$ so that for each $s \geq t_0$ if $\|u_{s} - x(s;x_0,t_0)\| <\delta$, then $\|x(t;u_s,s)-x(t;x_0,t_0)\|  \leq Ke^{-\gamma(t-s)}\|u_s-x(s;x_0,t_0)\|$ for all $t \geq s$.
\end{definition}

Using \eqref{eq:lineq} we can express the differential equation of \eqref{eq:odeivp} in linear inhomogeneous form $\dot{x} = f(x,t) = A(t)x + N(x,t)$.  If $x(t;x_0,t_0)$ is bounded in $t$, then an argument using the nonlinear variation of parameters formula shows that uniform exponential stability of \eqref{eq:lineq} implies uniform exponential stability of $x(t;x_0,t_0)$.  A similar implication is not true if \eqref{eq:lineq} is exponentially stable, but not uniformly so (see \cite{Perron1930} or Equation 14 in \cite{LK2007}).

 %Under the generic assumption (see page 21 of \cite{Palmer}) that \eqref{eq:lineq} has an integral separation structure the stability of $x(t;x_0)$ is determined by the stability of the zero solution of \eqref{eq:lineq}. 

The QR theory for the approximation of the Lyapunov and Sacker-Sell spectrum of \eqref{eq:lineq}, developed and analyzed extensively in \cite{DVV2003b,DVV2005,DVV2006,DVV2003}, is based on the construction of a time-dependent orthogonal change of variables.  Let $Q(t)$ be a solution of the differential equation
\begin{equation}\label{eq:Qeqn}
\dot{Q}(t) = Q(t)S(Q(t),A(t)), \quad S(Q,A)_{ij} = \left\{
\begin{array}{rl}
(Q^T AQ)_{i,j},& i>j\cr
0, & i=j\cr
-(Q^T A Q)_{i,j},& i<j\cr
\end{array}
\right.
\end{equation}
that satisfies $Q(t)^T Q(t) = I_d$ where $A(t)$ is the coefficient matrix of \eqref{eq:lineq}.  Under the change of variables $x=Q(t)y$, the system 
\begin{equation}\label{eq:Beqn}
\dot{y} = B(t)y, \quad B(t)=Q(t)^T A(t) Q(t) - Q(t)^T \dot{Q}(t), \quad t > t_0
\end{equation}
is such that $B(t)$ is upper triangular for all $t > t_0$.  We refer to the system \eqref{eq:Beqn} as a corresponding upper triangular system to \eqref{eq:lineq} (with initial orthogonal factor $Q(t)$) and the Lyapunov and Sacker-Sell spectra of these two systems coincide.  The following definition and theorem summarize how to compute end-points of the Lyapunov and Sacker-Sell spectrum of \eqref{eq:lineq} in terms of the diagonal entries of $B(t)$.

\begin{definition}\label{def:iss}
Assume that $B:(t_0,\infty) \rightarrow \mathbb{R}^{d\times d}$ is bounded, continuous, and upper triangular and let $B_{i,j}(t)$ denote the $i,j$ entry of $B(t)$.  Suppose that for any $i <j$ one of the two following conditions hold:
\begin{enumerate}
\item $B_{i,i}$ and $B_{j,j}$ are integrally separated, that is, there exists $a_{i,j} > 0$ and $b_{i,,j} \in \mathbb{R}$ so that if $t \geq s > t_0$, then
\begin{equation}\label{suffcondstablyap.0}
\int_{s}^{t}B_{i,i}(\tau) - B_{j,j}(\tau)d\tau \geq a_{i,j}(t-s)+b_{i,j}.
\end{equation}
\item For every $\varepsilon > 0$ there exists $M_{i,j}(\varepsilon) > 0$ so that if $t \geq s > t_0$, then
\begin{equation}\label{suffconstablyap.1}
\left|\int_{s}^{t}B_{i,i}(\tau)-B_{j,j}(\tau) d\tau\right| \leq M_{i,j} + \varepsilon(t-s).
\end{equation}
\end{enumerate}
Then we say that the ODE $\dot{y}=B(t)y$ and $B(t)$ have an integral separation structure.  If the first condition is satisfied for all $i < j$, then we say that $B(t)$ and $\dot{y}=B(t)y$ are integrally separated.   If the system \eqref{eq:lineq} has a corresponding upper triangular system that has an integral separation structure, then we say that \eqref{eq:lineq} has an integral separation structure and if the corresponding upper triangular system is integrally separated, then we say that \eqref{eq:lineq} is integrally separated.
\end{definition}

%The endpoints of the Sacker-Sell spectral intervals of \eqref{eq:lineq} can always be determined from the diagonal entries of $B(t)$ so long as $A(t)$ is bounded and continous (see e.g. Theorem 5.5 of \cite{DVV2003}).  Under the generic assumption (see page 21 of \cite{Palmer}) that \eqref{eq:Beqn} has an integral separation structure the endpoints of the Lyapunov spectral intervals of \eqref{eq:lineq} can be determined from the diagonal elements of $B(t)$ (see e.g. Theorem 5.1 of \cite{DVV2003}).   when the system has an integral separation structure or is asymptotically contracting.

\begin{theorem}[Theorems 2.8, 5.1, 5.5, 6.1 of \cite{DVV2003}]\label{thm:lects}
Let $B:(t_0,\infty) \rightarrow \mathbb{R}^{d\times d}$ be bounded, continuous, and upper triangular and let $\Sigma_{ED} = \cup_{i=1}^{d}[\alpha_i,\beta_i]$ denote the Sacker-Sell spectrum of $\dot{y}(t) = B(t)y(t)$.  For $i=1,\hdots,d$ we have:
\begin{equation}\label{eq:SSformula}
\alpha_i = \liminf_{H\rightarrow 0}\inf_{t \geq t_0}\dfrac{1}{H}\int_{t}^{t+H}B_{i,i}(\tau)d\tau, \quad \beta_i = \limsup_{H\rightarrow \infty}\sup_{t \geq t_0}\dfrac{1}{H}\int_{t}^{t+H}B_{i,i}(\tau)d\tau.
\end{equation}
%Furthermore, there exists $H > 0$ so that if $t-s > H$, then for $i=1,\hdots,d$ we have
%\begin{equation}\label{eq:SSformula2}
%\alpha_i = \inf_{t \geq t_0}\dfrac{1}{t-s}\int_{s}^{t}B_{i,i}(\tau)d\tau, \quad \beta_i = sup_{t \geq t_0}%\dfrac{1}{t-s}\int_{s}^{t}B_{i,i}(\tau)d\tau. 
%\end{equation}
Assume that $B:(t_0,\infty) \rightarrow \mathbb{R}^{d\times d}$ has an integral separation structure and let $\Sigma_L = \cup_{i=1}^{d}[\eta_i,\mu_i]$ denote the Lyapunov spectrum of $\dot{y}=B(t)y$.  Then the Lyapunov spectrum of $\dot{y}(t)=B(t)y(t)$ is continuous with respect to $L^{\infty}(t_0,\infty)$ perturbations of $B(t)$ and for $i=1,\hdots,d$:
\begin{equation}\label{suffconstablyap.2}
\eta_i =\liminf_{t\rightarrow \infty}\dfrac{1}{t-t_0}\int_{t_0}^{t}B_{i,i}(\tau)d\tau, \quad \mu_i =\limsup_{t\rightarrow \infty}\dfrac{1}{t-t_0}\int_{t_0}^{t}B_{i,i}(\tau)d\tau.
\end{equation}\qed
\end{theorem}
A similar theorem can be proved for discrete-time linear systems (see Section 3.2 of \cite{VV1} and Corollary 3.25 of \cite{Potzsche2012}) which is used to prove the main linear stability results of Section 3.1 in \cite{SVV1} and Section 3.2 of \cite{steyerthesis} which are summarized in Theorem \ref{thm:onesteptheorem} below.  This result is fundamentally based on the observation that the numerical solution of \eqref{eq:lineq} by a one-step method with local truncation error of order $p \geq 1$ takes the form $x_{n+1} = \Phi^A(n;h)x_n$.  For such a discrete-time difference equation we have for each fixed initial orthogonal $Q_0 \in \mathbb{R}^{d\times d}$ a discrete QR iteration $\Phi^A(n;h) Q_n = Q_{n+1} R^A(n;h)$ where $Q_n \in \mathbb{R}^{d\times d}$ is orthogonal and $R^A(n;h)$ is upper triangular with positive diagonal entries. 
\begin{theorem}\label{thm:onesteptheorem}
Let $x_{n+1} = \Phi^A(n;h)x_n$ denote the numerical solution to \eqref{eq:lineq} by a one-step method with local truncation error of order $p \geq 1$ with step-size $h > 0$ and initial condition $x_0$.  Let $\Sigma_L^A = \cup_{i=1}^{n}[\eta_i^A,\mu_i^A]$ and $\Sigma_{ED}^A = \cup_{i=1}^{d}[\alpha_i^A,\beta_i^A]$ denote the Lyapunov and Sacker-Sell spectrum respectively of the discrete nonautonomous difference equation $x_{n+1} = \Phi^A(n;h)x_n$ and let $\Sigma_L= \cup_{i=1}^{n}[\eta_i,\mu_i]$ and $\Sigma_{ED} = \cup_{i=1}^{d}[\alpha_i,\beta_i]$ denote the Lyapunov and Sacker-Sell spectrum of \eqref{eq:lineq}.
\begin{enumerate}
\item If the coefficient matrix $A(t)$ of \eqref{eq:lineq} is bounded and continuous, then for every $\varepsilon > 0$ there exists $h^* > 0$ so that if $h \in (0,h^*)$, then $|\alpha_i^A-\alpha_i| < \varepsilon$ and $|\beta_i^A -\beta_i| < \varepsilon$ for $i=1,\hdots,d$.

\item Assume \eqref{eq:lineq} has an integral separation structure and let $Q_0 \in \mathbb{R}^{d \times d}$ be orthogonal.  Let $R^A(n;h)$ be the corresponding upper triangular factor of the discrete QR iteration applied to $x_{n+1} = \Phi^A(n;h)x_n$ with initial orthogonal factor $Q_0$ and let $B(t) = Q(t)^T A(t) Q(t) - Q(t)^T \dot{Q}(t)$ where $Q(t_0) = Q_0$ is the unique orthogonal solution of \eqref{eq:Qeqn}.  There exists $h^* > 0$ so that if $h \in (0,h^*)$ and $i=1,\hdots,d$, then $|\alpha_i^A-\alpha_i|,|\eta_i^A-\eta_i|,|\beta_i^A -\beta_i|,|\mu_i^A-\mu_i|= \mathcal{O}(h^{p})$ and for $n \geq 0$ the diagonal entries of $R^A(n;h)$ satisfy that $\ln(R^A_{i,i}(n;h)) = \int_{t_n}^{t_{n+1}}B(\tau)d\tau + \mathcal{O}(h^{p+1})$.

\qed
\end{enumerate}
\end{theorem}
We apply Theorem \ref{thm:onesteptheorem} in Section \ref{sec:mainresults2} to prove a stability result for strictly stable GLMs solving nonautonomous, linear ODEs with Sacker-Sell spectrum bounded above by zero.  Similar to results derived in \cite{Eirola1988} and \cite{Beyn1987}, an argument with the variation of parameters formula allows us to use Theorem \ref{thm:onesteptheorem} to prove the following theorem (See Theorems 4.1-2 in \cite{SVV1} or Theorem 9 in \cite{steyerthesis}).
\begin{theorem}\label{thm:onestep_nonlin}
Assume that $x(t;x_0,t_0)$ is bounded and the right end-point of the Sacker-Sell spectrum of \eqref{eq:lineq} is $-\alpha < 0$.  Let $u(n;u_{0},h)$ denote the numerical solution by a one-step method with local truncation error of order $p \geq 1$ of $\dot{x} = f(x,t)$ with the initial condition $u_{0}$ at initial time $t_0 > \tau_0$ and fixed step-size $h > 0$.  Suppose that $f \in C^{p+2}(\mathbb{R}^d \times (\tau_0,\infty))$ and $f(x,\cdot)$ is bounded for each $x \in \mathbb{R}^d$.  Then given any $D > 0$ and $\gamma \in (0,\alpha)$, there exists $h^* > 0$, $K > 0$ and $0 < \delta_2 < \delta_1$ so that if $\|u_{0}-x_0\| < \delta_1$ and $h \in (0,h^*)$, then $\|u(n;u_{0},h)-x(nh+t_0;u_0,t_0)\| \leq D h^{p}$ and if $\|u_{0}-x_0\| < \delta_2$, $h \in (0,h^*)$, and $n \geq m$, then 
$$
\|u(n;u_{0},h)-u(n;x_0,h)\| \leq K  e^{-\gamma h(n-m)}\|u(n;u_{0},h)-u(n;x_0,h)\|.
$$\qed
\end{theorem}
In Section \ref{sec:mainresults2} we explain how to obtain an analogous nonlinear stability theorem for GLMs using Proposition \ref{thm:nonlinapprox}.

\section{Main Results}\label{sec:mainresults}

\subsection{Nonautonomous underlying one-step methods}\label{sec:mainresults1}

In this section we prove that there exists a unique underlying one-step method for a strictly stable GLM approximating the solution of a nonlinear and nonautonomous ODE whose nonlinear part satisfies a global Lipschitz condition.  Throughout we consider a strictly stable, $k$-step, and $r$-stage GLM \eqref{eq:glm} that we denote by $\mathcal{M}$ which we assume has local truncation error of order $p \geq 1$.  We let $P \in \mathbb{R}^{k \times k}$ be a matrix so that $E = P^{-1}VP$ is of the form $E = \left[\begin{array}{cc} 1 & 0 \\ 0 & E_{2,2} \end{array}\right]$ 
where the eigenvalues of $E_{2,2} \in \mathbb{R}^{k-1 \times k-1}$ all have modulus strictly less than $1$ ($E$ may be taken to be e.g. the real Jordan form of $V$).  The main result of this section is the following theorem. 
\begin{theorem}\label{thm:maintheorem1}
Consider the following ODE
\begin{equation}\label{eq:lininhomo2}
\dot{x}=f(x,t) = A(t)x + N(x,t)
\end{equation}
where $f:\mathbb{R}^{d} \times (\tau_0,\infty) \rightarrow \mathbb{R}^d$, $\tau_0 \geq -\infty$, and $N(x,t)$ satisfies the global in space Lipschitz condition that there exists $K > 0$ so that for all $x,y \in \mathbb{R}^d$ and $t > \tau_0$ we have
\begin{equation}\label{eq:lip}
\|N(x,t)-N(y,t)\| \leq K\|x-y\|.
\end{equation}
Assume that $A(t)$ is bounded, $f(x,t)$ is $C^{p+1}$ on its domain, and the partial derivatives  $\dfrac{\partial^k f}{\partial x^k}(x,t)$ are bounded for $k=1,\hdots,p+1$.  Let $\{X_n\}_{n=0}^{\infty}$ where $X_n = X(n;X_0,t_0,h)$ denotes the output of $\mathcal{M}$ applied to solve \eqref{eq:lininhomo2} using step-size $h > 0$, initial value $X_0 \in \mathbb{R}^{dk}$, and initial time $t_0 >\tau_0$.  Then there exists $G> 0$, $\gamma \in (0,1)$, and $h^* > 0$ such that the following conclusions hold for any $t_0 > \tau_0$, $X_0 \in \mathbb{R}^d$, and $h \in (0,h^*)$.
\begin{enumerate}
\item If $h \in (0, h^*)$, then $\{X_n\}_{n=0}^{\infty}$ is the solution of a nonautonomous discrete dynamical system $X_{n+1} = F(X_n,n,h)$.

\item The difference equation
$$Y_{n+1} = H(Y_n,n,h) \equiv (P^{-1}\otimes I_d)F((P\otimes I)Y_n,n,h)$$
 defined from the change of variables $X_n = (P \otimes I)Y_n$ satisfies that if $h \in(0, h^*)$, then there exists a unique (but see the remark immediately after the statement of this theorem) continuous function $\varphi: \mathbb{R}^d \times \mathbb{Z} \times (0,h^*) \rightarrow \mathbb{R}^{d(k-1)}$ whose graph is invariant under the flow of $Y_{n+1} = H(Y_n,n,h)$ for $n \geq 0$ and such that for any $Y_0 \in \mathbb{R}^{dk}$, there exists $z_0^1 \in \mathbb{R}^d$ such that the solution $\{Y_n\}_{n=0}^{\infty}$ of $Y_{n+1} = H(Y_n,n,h)$ for $n \geq 0$ using initial value $Y_0$ satisfies 
\begin{equation}\label{eq:maintheorem1est}
\|Y_n - Z_n\| \leq G\gamma^n, \quad n \geq 0
\end{equation} 
where the sequence $\{Z_n\}_{n=0}^{\infty}$ is such that $Z_n = ((z_n^1)^T,\varphi(z_n^1,0,h)^T)^T$ and $Z_{n+1} = H(Z_n,n,h)$ for all $n \geq 0$.

\item If $h \in (0,h^*)$, then $\varphi(y,m)$ is globally Lipschitz and $C^{p+1}$ with respect to the state variable $y$.

\item  The difference equation $y_{n+1} = H_1(y_n,\varphi(y_n,0,h),n,h)$ where $H_1$ denotes the first $d$ components of $H$ defines a unique one-step method referred to as the underlying one-step method.  Let $\mathcal{F}^*_h$ be a finishing procedure defined by projecting a vector $Y :=(y^1,y^2) \in \mathbb{R}^{d} \times \mathbb{R}^{d(k-1)}$ onto its first $d$ components using the formula $\mathcal{F}_h^*(Y) = y^1$.  If $h \in (0,h^*)$, then for each underlying one-step method there exists a starting procedure $\mathcal{S}_h^*$ that takes the form $\mathcal{S}_h^*(x) = (x^T,\varphi(x,0,h)^T)^T$ such that the method $\mathcal{M}$ is of order $p$ relative to $\mathcal{S}_h^*$ and so that for any $x \in \mathbb{R}^d$ and $n \geq 0$ we have $\mathcal{S}_h^*(H_1(x,\varphi(x,0,h),n)) = H(\mathcal{S}_h^*(x),n,h)$ and $H_1(x,\varphi(x,0,h),n,h) = \mathcal{F}_h^*(H (\mathcal{S}_h^*(x),n,h))$.
\end{enumerate}
\end{theorem}
We first remark that $\varphi$ and the underlying one-step method are not uniquely defined unless we agree to extend the difference equations defined by applying $\mathcal{M}$ to solve \eqref{eq:lininhomo2} from $\mathbb{N}$ to $\mathbb{Z}$ in a unique way.  This is because the uniqueness of the function $\varphi$ whose graph defines the pseudo-unstable manifold of a nonautonomous difference equation satisfying a gap condition relies on the difference equation being defined on all of $\mathbb{Z}$ rather than merely $\mathbb{N}$ (see Theorem 4.1 of \cite{aulbach1998}).

We also remark that the assumption that $N(x,t)$ satisfies the global Lipschitz condition \eqref{eq:lip}, which is quite strong, is not essential for our results and is used for simplicity.  In general (see Remark 2.7 (2) of \cite{PotzscheRasmussen2008}) all we need is that $(N(y,t)-N(x,t))/\|y-x\| \rightarrow 0$ as $y \rightarrow x$ uniformly for $t \geq t_0$.

The remainder of this section is dedicated to the proof of Theorem \ref{thm:maintheorem1}.  Let $t_0 > \tau_0$ and $X_0 \in \mathbb{R}^{dk}$.  The method $\mathcal{M}$ applied to solve \eqref{eq:lininhomo2} with step-size $h > 0$ using the initial value $X_0$ at the initial time $t_0 > s$ takes the form 
\begin{equation}\label{eq:glmintstage}
\left\{\begin{array}{lcr}
G_n = (U \otimes I_d)X_n +h(C \otimes I_d)M_n G_n  + h (C \otimes I_d) \overline{N}_n\\
X_{n+1} = (V \otimes I_d)X_n +h(D \otimes I_d)M_n G_n +h(D \otimes I_d)\overline{N}_n
\end{array}\right.
\end{equation}
where $M_n = \text{diag}(A_{n,1},\hdots,A_{n,r}) \in \mathbb{R}^{dr \times dr}$, \newline
${\overline{N}_n =(N(g_{n,1},t_n+\xi_1 h)^T,\hdots,N(g_{n,r},t_n+\xi_r h)^T)^T}$, and $A_{n,i} = A(t_n+\xi_ih)$ for $i=1,\hdots,r$ where $t_n :=t_0 + nh$.  The equation \eqref{eq:glmintstage} implies that the internal stages $G_n$ satisfy the following algebraic condition
\begin{equation}\label{eq:exstage}
G_n = [I_{dr}-h(C \otimes I_d) M_n]^{-1}(U \otimes I_d)X_n + h[I-h(C \otimes I_d) M_n]^{-1}(C \otimes I_d)\overline{N}_n.
\end{equation}
The implicit function theorem and the fact that $f(x,t)=A(t)x+N(x,t)$ is at least $C^2$ (since $p \geq 1$) then implies that there exists $h^* > 0$ so that $h \in (0,h^*)$, then 
\begin{equation}\label{eq:glmdiffeqn}
X_{n+1} =  (V \otimes I_d)X_n  + R(X_n,n,h)
\end{equation}
where (because of \eqref{eq:lip}) the term $R(X,t,h)$ is Lipschitz in $X_n$ with Lipschitz constant $L_R=L_R(h)$ bounded as $L_R(h) \leq h J'$ for some constant $J' > 0$.  Therefore the first conclusion of Theorem \ref{thm:maintheorem1} is proved.  If we write $Y_n=((y_n^1)^T,(y_n^2)^T)^T$ where $y_n^1 \in \mathbb{R}^d$ and $y_n^2 \in \mathbb{R}^{d(k-1)}$, then under the change of variables $X_{n} = (P \otimes I_d)Y_n$ the resulting system $Y_{n+1} := H(Y_n,n,h)$ can be expressed as 
\begin{equation}\label{eq:glmlineq}
\left\{\begin{array}{lcr}
y_{n+1}^1 = y_n^1 + R_1(Y_n,n,h) \\
y_{n+1}^2 = (E_{2,2} \otimes I_d) y_n^2 + R_2(Y_n,n,h) 
\end{array}\right.
\end{equation}
where $R_1$ and $R_2$ each have Lipschitz constants $L_{R_1} = L_{R_1}(h)$ and $L_{R_1}=L_{R_2}(h)$ bounded by $hJ$ where $J \leq \|P^{-1} \otimes Id\|J' \|P\otimes I_d\|$.  The following is an invariant manifold theorem for difference equations of the form \eqref{eq:glmlineq} and is a restatement of the conclusions of Theorem 3.1, Theorem 3.2, and Theorem 5.1 in \cite{ARS2005a} (See also \cite{ARS} and \cite{aulbach1998}).  It is included for completeness.
\begin{theorem}\label{thm:naim}
Consider a system of difference equations of the form
\begin{equation}\label{eq:glmlineqform}
\left\{\begin{array}{lcr}
x_{n+1} = A_n x_n + F_1(n,x_n,y_n)\\
y_{n+1} = B_n y_n + F_2(n,x_n,y_n)
\end{array}\right., \quad n \in \mathbb{Z}
\end{equation}
where $A_n \in \mathbb{R}^{d_1 \times d_1}$, $B_n \in \mathbb{R}^{d_2 \times d_2}$, and $F_i:\mathbb{Z}\times \mathbb{R}^{d_1} \times \mathbb{R}^{d_2} \rightarrow \mathbb{R}^{d_i}$ for $i=1,2$ where 
\begin{equation}\label{eq:lingap}
\begin{array}{lcr}
\|\prod_{j=n}^{m} A_j^{-1}\| \leq K \beta^{n-m}, \quad  n \leq m\\
\|\prod_{j=m}^{n} B_j\| \leq K \alpha^{n-m}, \quad n \geq m 
\end{array}
\end{equation}
and
\begin{equation}
\begin{array}{lcr}
\|F_1(n,x_n,y_n)-F_2(n,\tilde{x}_n,\tilde{y}_n)\| \leq L\|x_n-\tilde{x}_n\| + L \|y_n-\tilde{y}_n\| \\
\|F_2(n,x_n,y_n)-F_2(n,\tilde{x}_n,\tilde{y}_n\| \leq  L\|x_n-\tilde{x}_n\| + L \|y_n-\tilde{y}_n\|
\end{array}
\end{equation}
for constants $L > 0$, $K \geq 1$ and $0 < \alpha < \beta$ satisfying the following conditions
\begin{equation}\label{eq:specgap}
0 < L < \frac{\beta-\alpha}{4K}(2+K - \sqrt{4+K^2}), \quad c(\alpha+2KL)  < 1 < c(\beta-2KL)
\end{equation}
for some $c > 0$.  Denote the solution of \eqref{eq:glmlineqform} with the initial condition $z_m=\left[\begin{array}{c} x_m \\y_m\end{array} \right]$ at initial time $m$ as
\begin{equation}\label{eq:glmlineqsol}
z(n;m,x_m,y_m) = \left[\begin{array}{c} x(n;m,x_m,y_m) \\y(n;m,x_m,y_m)\end{array}\right]
\end{equation}
Then there exists a unique continuous map $\varphi:\mathbb{R}^{d_1}\times \mathbb{Z}  \rightarrow \mathbb{R}^{d_2}$ whose graph is the manifold 
$$\mathcal{D}=\{(m,x,\varphi(x,m)): m \in \mathbb{Z}, x\in \mathbb{R}^{d_1}\}$$
and $\mathcal{D}$ is invariant under the discrete flow of \eqref{eq:glmlineqform}.  Additionally, $\mathcal{D}$ is globally exponentially attracting in the sense that for any $m \in \mathbb{Z}$, $z_m= (x_m,y_m) \in \mathbb{R}^{d_1} \times \mathbb{R}^{d_2}$ there exists $(m,w_m,\varphi(w_m,m)) \in \mathcal{D}$, $G > 0$ and $\gamma \in (0,1)$ so that 
\begin{equation}
\|z(n;m,x_m,y_m)-z(n;m,w_m,\varphi(w_m,m))\| \leq G\gamma^{n-m}, \quad n \geq m
\end{equation} \qed
\end{theorem}
We use Theorem \ref{thm:naim} to complete the proof of Theorem \ref{thm:maintheorem1}.  There exists $h_1^* > 0$ so that if $h \in (0,h_1^*)$, then $X_n$ satisfies the difference equation \eqref{eq:glmdiffeqn}.  The matrix sequence $\{Y_n\}_{n=0}^{\infty}$ where $Y_n = (P^{-1} \otimes I_d)X_n$ satisfies the difference equation \eqref{eq:glmlineq}.  If \eqref{eq:lininhomo2} is not defined on all of $\mathbb{Z}$ (i.e. $s > -\infty)$, then we uniquely extend the difference equation \eqref{eq:glmlineq} satisfied by $\{Y_n\}_{n=0}^{\infty}$ that is defined on $\mathbb{N}$ to a difference equation defined on all of $\mathbb{Z}$ by setting $R_i(\cdot,n,\cdot) \equiv 0$ for $i=1,2$ whenever $n < 0$.  Since the eigenvalues of $E_{2,2}$ all have modulus strictly less than $1$ this extended difference equation on $\mathbb{Z}$ is of the form \eqref{eq:glmlineqform} for $\alpha < \beta = 1$ and $L=hJ$. Thus, we can choose $c > 0$ and $h^* \in (0,h_1^*]$ so that the inequalities \eqref{eq:specgap} are satisfied whenever $h \in (0, h^*)$.  So, there exists a continuous map $\varphi:\mathbb{R}^d \times\mathbb{Z} \times (0,h^*) \rightarrow \mathbb{R}^{d(k-1)}$ that is invariant under the flow of $H$ and such that if $Y_n$ is the solution of \eqref{eq:glmlineq}, then there exists a sequence $\{z_n^1\}_{n=0}^{\infty}$ with $z_n^1 \in \mathbb{R}^{d}$, $G > 0$, and $\gamma\in (0,1)$  such that
\begin{equation}\label{eq:onestep_fin}
\begin{array}{c}
z_{n+1}^1 = z_{n}^1 + R_1(z_n^1,\varphi(z_n^1,0,h),n,h) \equiv H_1(z_n^1,n,h)\\
  \|Y_n-(z_n^1,\varphi(z_n^1,0,)^T)^T\| \leq G \gamma^{n}, \quad n \geq 0.
\end{array}
\end{equation}
This completes the proof of the second conclusion of Theorem \ref{thm:maintheorem1}.  Conclusion 3 follows from the results of \cite{AWP2002}.  The fourth conclusion is proved by repeating the proof of Theorem 2.3 in \cite{Stoffer1993} using the function $\varphi(y,0)$, Conclusion 3, and the definition of local truncation error for GLMs. \qed

\subsection{Nonautonomous stability of general linear methods}\label{sec:mainresults2}

In this section we combine the results of Theorems \ref{thm:onesteptheorem} and \ref{thm:maintheorem1} to prove a stability result for the solution of a nonautonomous linear ODE by a strictly stable GLM.  Consider the ODE
\begin{equation}\label{eq:lin}
\dot{x} = A(t)x, \quad t > \tau_0
\end{equation}
where $A:(t_0,\infty) \rightarrow \mathbb{R}$ and $\tau_0 \geq -\infty$.  The following theorem states that if the step-size of a GLM satisfying the hypotheses of Theorem \ref{thm:maintheorem1} solving the linear ODE \eqref{eq:lin} is sufficiently small, then the exponential stability/instability of numerical solutions of \eqref{eq:lin} found with the GLM are determined by the stability spectra its underlying one-step method approximates.  Notice, however, that we are unable to show uniform exponential stability/instability of the solution found with the GLM.
\begin{theorem}\label{thm:maintheorem2}
Suppose that the coefficient matrix $A(t)$ of the nonautonomous linear ODE \eqref{eq:lin} is bounded and $C^{p+1}$.  Assume that the method \eqref{eq:glm} denoted by $\mathcal{M}$ is strictly stable and has local truncation error of order $p \geq 1$.  Let $X_n:=X(n;X_0,t_0,h)$ denote the output of $\mathcal{M}$ applied to solve \eqref{eq:lin} using step-size $h > 0$, initial time $t_0 > \tau_0$, and initial value $X_0 \in \mathbb{R}^{dk}$.  Denote the Sacker-Sell spectrum of \eqref{eq:lin} by $\Sigma_{ED}$. 
\begin{enumerate}
\item If $\Sigma_{ED} \cap [0,\infty) = \emptyset$, then for each initial value $X_0$ there exists $h^* > 0$, $G > 0$, and $\gamma \in (0,1)$ so that if $h \in (0,h^*)$, then  $\|X(n;X_0,t_0,h)\| \leq G \gamma^{n}$.  

\item If $\Sigma_{ED} \cap [0,\infty) \neq \emptyset$, then there exists $h^* > 0$, $G > 0$, and $\gamma > 1$ so that if $h \in (0,h^*)$, then  $\|X(n;X_0,t_0,h)\| \geq G \gamma^{n}$ for some initial value $X_0$. 

\end{enumerate}
An analogous result holds for the Lyapunov spectrum of \eqref{eq:lin} if we assume that the ODE has an integral separation structure. 
\end{theorem}
\begin{proof}
We prove the first conclusion since the proof of the second is very similar.  Let $X_0 \in \mathbb{R}^{dk}$ be some initial condition at the fixed initial time $t_0$.  Since $A(t)$ is bounded and $C^{p+1}$ and $\mathcal{M}$ is strictly stable and has local truncation error of order $p \geq 1$, we can choose $h_1^* > 0$ so small that the four conclusions of Theorem \ref{thm:maintheorem1} hold for $h \in (0,h_1^*)$.  The first conclusion of Theorem \ref{thm:maintheorem1} implies that $X_{n+1} = F(X_n,n,h)$ for some function $F$ and the second conclusion of \ref{thm:maintheorem1} implies that there exists $G_1 > 0$, $\gamma_1 \in (0,1)$, and $\varphi:\mathbb{R}^{d}\times \mathbb{Z} \times (0,h_1^*) \rightarrow \mathbb{R}^{d(k-1)}$ so that 
$$\|(P^{-1} \otimes I_d)X_n - Z_n\| \leq G_1 \gamma_1^n, \quad n \geq 0$$
where $P$ is as defined in Section \ref{sec:mainresults1} and $Z_n = ((z_n^1)^T,(\varphi(z_n^1,0,h))^T)^T$ is a solution of $Z_{n+1} = H(Z_n,n,h)$ with $H$ and $\varphi$ defined as in Theorem \ref{thm:maintheorem1}.  The fourth conclusion implies that $z_{n+1}^1 = H_1(z_n^1,\varphi(z_n^1,0,h),n,h)$, where $H_1$ is the first $d$ components of $H$, defines a one-step approximation with local truncation error of order $p$ to $\dot{x}=A(t)x$ with initial condition $z_0^1$.  We therefore can write $z_{n+1}^1 = H_1(z_n^1,\varphi(z_n^1,0,h),n,h) \equiv \Phi^A(n;h)z_n$.  Theorem \ref{thm:onesteptheorem} then implies that there exists $h_2^*  \in (0,h_1^*]$ so that if $h \in (0,h_2^*)$ then the Sacker-Sell spectrum of $z_{n+1} = \Phi^A(n;h)z_n$ is bounded above by zero and therefore  
\begin{equation}\label{eq:maintheorem2.1}
\|z_n^1 \| \leq G_2 \gamma_2^{n-m}\|z_m^1\|, \quad n \geq m \geq 0
\end{equation}
for some $G_2 > 0$ and $\gamma_2 \in (0,1)$.  By the work in the previous section, there exists $h_3^* \in (0,h_2^*]$ so that if $h \in (0, h_3^*)$, then $F(X_n,n,h) = \Phi(n,h)X_n$ and $H(Y_n,n.h) = (P^{-1} \otimes I_d)\Phi(n;h)(P \otimes I_d)Y_n$ where
$$\Phi(n;h) = (V \otimes I_d) + h (D\otimes I_d)M_n[I-h (C \otimes I_d) M_n]^{-1}$$
and $\Phi(n;h)$ is bounded and invertible with $M_n$ as defined in Equation \eqref{eq:glmlineq}.  The third conclusion of Theorem \ref{thm:maintheorem1} implies that there exists $h_3^* \in (0,h_2^*]$, $G_3 > 0$, and $\gamma_3 \in (0,1)$ so that if $h \in (0,h_3^*)$, then 

% add why the third conclusion implies this as a comment

$$\|Z_n\| \leq G_3 \gamma_3^{n-m}\|Z_m\|, \quad n \geq 0.$$
Take $h^* = \text{min}\{h_1^*,h_2^*,h_3^*\}$.  If $h \in (0,h^*)$ and $n \geq 0$, then 
$$\|X_n \| \leq \|(P\otimes I_d)\|\left( \|(P^{-1}\otimes I_d)X_n-Z_n\| + \|Z_n\|\right) \leq \|(P\otimes I_d)\|\left(G_1 \gamma_1^n + G_3\|Z_0\| \gamma_3^n\right).$$
The result follows by taking $G = \|P\otimes I_d\|\text{max}\{G_1,G_3\|Z_0\|\}$ and $\gamma = \text{max}\{\gamma_1,\gamma_3\}$.
\end{proof}
Various types of scalar test equations are often used to characterize the stability properties of GLMs solving ODE IVPs.  In \cite{steyerthesis} and \cite{SVV1} it is shown that the stability of the numerical solution  by a one-step method with local truncation error of order $p \geq 1$ of a nonautonomous linear ODE with a bounded and sufficiently smooth coefficient matrix can be approximately characterized by the one-step method applied to $d$ scalar test equations of the form 
\begin{equation}\label{eq:testeq2}
\dot{x}=\lambda(t)x, \quad t > t_0
\end{equation}
where $\lambda:(t_0,\infty)\rightarrow \mathbb{R}$ is the real-valued diagonal element of a matrix $B(t)$ of a corresponding upper triangular system $\dot{y}=B(t)y$ to \eqref{eq:lin}.  Theorem \ref{thm:maintheorem2} justifies using such test equations to characterize the stability of strictly stable GLMs solving nonautonomous linear ODEs by passing to the approximation properties of the underlying one-step method.

Theorem \ref{thm:maintheorem2} is an asymptotic result showing that as $h \rightarrow 0$ we can guarantee the exponential decay of the numerical solution of a nonautonomous linear ODE whose Lyapunov or Sacker-Sell spectrum lies to the left of zero.  It is natural to look for a subset of A-stable methods that preserve the asymptotic decay of all such linear ODEs with no restriction on $h$.  The following theorem partially answers this question and says that step-size restriction is essential for the preservation of asymptotic decay by strictly stable linear multistep and Runge-Kutta methods.
\begin{theorem}\label{thm:counterexample}
Given any strictly stable and consistent linear multistep method or convergent Runge-Kutta method $\mathcal{M}$ and $h > 0$, there exists a uniformly exponentially stable scalar ODE $\dot{x}=\lambda(t)x$ such that the numerical solution $\{x_n\}_{n=0}^{\infty}$ by $\mathcal{M}$ with step-size $h > 0$ becomes unbounded as $n \rightarrow \infty$ for any initial condition $x(0) = x_0 \neq 0$.
\end{theorem}
\begin{proof}
Let $\mathcal{S}$ denote the linear stability domain of $\mathcal{M}$ and let $\partial \mathcal{S}$ denote its boundary.  Since $\mathcal{M}$ is a strictly stable linear multistep method or a convergent Runge-Kutta method it follows that there exists $\delta > 0$ such that $(0,\delta) \notin \mathcal{S} \cup \partial \mathcal{S}$.  Consider the ODE $\dot{x}=(D\cos( \omega t)+L)x \equiv \lambda(t)x$ where $t > 0$, $0 < D +L < \delta/2$, $L < 0$ and $\omega = 2\pi / h$ and let $x(0)=x_0 \neq 0$.  Notice that $L < 0$ implies that zero is uniformly exponentially stable for $\dot{x} = \lambda(t)x$.  The equation $D\cos(\omega nh)+L = D+L$ implies that the solution of $\dot{x}=\lambda(t)x$ using $\mathcal{M}$ with step-size $h > 0$ is the same as the numerical solution of the ODE $\dot{x} = (D+L)x$.  The quantity $h(D+L) \notin \mathcal{S} \cup \partial \mathcal{S}$ since $h(D+L) < \delta /2$.  Therefore the numerical solution of $\dot{x}=\lambda(t)x$ by $\mathcal{M}$ using the step-size $h> 0$ becomes unbounded as $n \rightarrow \infty$.
\end{proof}
The geometric idea behind the proof of Theorem \ref{thm:counterexample} is that if the step-size is too large, then $h \lambda(nh+t_0)$ may be outside the classical stability domain too often and destabilize the numerical solution.  It seems impossible to devise general algebraic conditions on the method coefficients of \eqref{eq:glm} guaranteeing that the numerical solution of any uniformly exponentially stable test equation of the form \eqref{eq:testeq2} decays.  However, if $h\lambda(t)$ is in the linear stability region of \eqref{eq:glm} on average, then we can use the standard linear stability region techniques in an approximate sense as we now show.  For each $n \geq 0$ we have an associated mean test equations
\begin{equation}\label{eq:meantesteq}
\dot{w}_n = \xi(n;h)w_n, \quad \xi(n;h) = \frac{1}{h}\int_{nh}^{(n+1)h}\lambda(\tau)d\tau.
\end{equation}
For $n \geq 0$ the exact solutions of \eqref{eq:testeq2} and \eqref{eq:meantesteq} at $t = (n+1)h$ using the same initial condition $x(t_n)=x_n$ agree and are given by
$$w_n((n+1)h) = x((n+1)h) = \exp\left(\int_{nh}^{(n+1)h}\lambda(\tau)d\tau\right)x_n.$$
Assume that $h > 0$ is so small that an underlying one-step method of \eqref{eq:glm} exists.  The numerical solution of \eqref{eq:testeq2} by \eqref{eq:glm} with step-size $h > 0$ is given by $x_{n+1} = \Phi^{\lambda}(n;h)x_n$.  For $n \geq 0$ applying the underlying one-step method to compute one forward step of the numerical solution of \eqref{eq:meantesteq} with initial condition $w_n(nh)=x_n$ is given by $\Phi^{\xi(n;h)}(h)x_n$.  Assuming that \eqref{eq:glm} has local truncation error of order $p \geq 1$, it follows that there exists $h^* > 0$ so that if $h \in (0,h^*]$, then 
$$\Phi^{\lambda}(n;h) = \Phi^{\xi(n;h)}(h) + \mathcal{O}(h^{p+1}).$$
Thus, the nonautonomous stability of the test equation \eqref{eq:testeq2} over the interval $[nh,(n+1)h]$ is approximately determined by the time average \eqref{eq:meantesteq} of $\lambda(t)$ on $[nh,(n+1)h]$.  If $h\xi(n;h)$ is in the interior of linear stability region of the \eqref{eq:glm} for all sufficiently small $h > 0$, then the sequence $\{\Phi^{\xi(n;h)}(h)\}_{n=0}^{\infty}$ is power bounded (see Definition 2.1 of \cite{Butcher1987b}) and it follows that $\{\Phi^{\lambda}(n;h)\}_{n=0}^{\infty}$ is power bounded for all sufficiently small $h > 0$.  Hence, we can repeat the analysis of Runge-Kutta methods in Theorem 3.6 of \cite{SVV1} and in an approximate sense determine the stability of \eqref{eq:glm} applied to solve \eqref{eq:testeq2} from the application of the underlying one-step method to solve mean test equations of the form \eqref{eq:meantesteq}.

%A one-step method is said to be autonomous if whenever it is used to solve an autonomous ODE $\dot{x}= f(x)$ with constant step-size $h > 0$ the numerical solution is given by an autonomous difference equation $x_{n+1} = theta(x_n;h,f)$.  

%The numerical solution of \eqref{eq:testeq2} by a one-step method  $\mathcal{M}'$ using step-size $h > 0$ and initial value $x(t_0) = x_0$ takes the form $x_{n+1} = \Phi^{\lambda}(n;h)x_n$.  For
%If the one-step method $\mathcal{M}'$ is autonomous, then the numerical solution of the mean test problem by $\mathcal{M}'$ with step-size $h > 0$ takes the form 
%$$x_{k+1} = \theta(xi(n;h),h)x_k, k > 0$$
%In \cite{SVV1} it is shown in Theorem 3.5 that for an autonomous one-step method, 

%The stability 

%Analagous to the proof of Corollary 2.4 in \cite{Stoffer1993}, we can 

%  It is challenging to extend the technique developed in this section in a straightforward way to uniformly exponentially stable solutions of nonlinear ODE IVPs satisfying the hypotheses of Theorem \ref{thm:maintheorem1}.  The main barrier for extending our techniques is that Theorem \ref{thm:maintheorem2} only shows exponential stability (and not uniform exponential stability) of the solution of the GLM applied to solve \eqref{eq:lin}. 

% It is also hard to use the structure of the underlying one-step method in this analysis since it is unlikely that the underlying one-step method is a Runge-Kutta or other classical one-step method and indeed (see \cite{EN1988}) such abstractly defined methods can be 'quite exotic'.  

As a follow-up to our linear stability theory we prove the following proposition for nonlinear initial value problems. From this proposition it follows that the conclusion of Theorem \ref{thm:onestep_nonlin} holds for underlying one-step methods and thus for the approximation generated by a strictly stable GLM with suitably chosen starting and finishing procedures.
\begin{proposition}\label{thm:nonlinapprox}
Consider a GLM \eqref{eq:glm} that is strictly stable and has local truncation error of order $p \geq 1$.  Then for any underlying one-step method $y_{n+1} = H_1(y_n,\varphi(y_n,0,h),h)$ there exists a starting procedure $\overline{\mathcal{S}}_h$, a finishing procedure $\overline{\mathcal{F}}_h$, and an $h^* > 0$ so that if $h \in (0,h^*)$, then the output of the GLM is defined by the map $X_{n+1} = F(X_n,n,h)$ and for any $x \in \mathbb{R}^d$ and $n \geq 0$ we have
\begin{equation}\label{eq:nonlinapprox.0}
H_1(x,\varphi(x,0,h),n,h) = \overline{\mathcal{F}}_h(F(\overline{\mathcal{S}}_h(x),n,h)).
\end{equation}
\end{proposition}
\begin{proof}
Let $h^* > 0$ be such that if $h \in (0,h^*)$, then the conclusions of Theorem \ref{thm:maintheorem1} hold.  If $h \in (0,h^*)$, then the first conclusion implies that the output of the GLM satisfies $X_{n+1} = F(X_n,n,h)$.  Let $\overline{\mathcal{S}}_h := (P \otimes I_d)S_h^*$ and $\overline{\mathcal{F}}_h := \mathcal{F}_h^* \circ (P^{-1} \otimes I_d)$ where $P$ is as defined in Section \ref{sec:mainresults1}.  Then \eqref{eq:nonlinapprox.0} follows by combining $H(Y,n,h) = (P^{-1} \otimes I_d)F((P^{-1}\otimes I_d)Y,n,h)$ with the fourth conclusion of Theorem \ref{thm:maintheorem1}.
\end{proof}
In general we will not have explicit formulas for the starting procedure in Proposition \ref{thm:nonlinapprox} since it is defined in terms of the map $\varphi$.  However, as shown in the proof of Theorem 2.3 in \cite{Stoffer1993}, if we have a starting procedure $\mathcal{S}_h$ relative to which a strictly stable GLM has local truncation error of order $p$, then we can show that $\mathcal{S}^*_h = \mathcal{S}_h + \mathcal{O}(h^{p+1})$.  This is sufficient for the output of a GLM to be (non-uniformly) exponentially attracted to a uniformly exponentially stable trajectory since the one-step method is exponentially attractive.

\section{Experiments}\label{sec:experiments}

In this section we develop a stability diagnostic for strictly stable GLMs solving time-dependent linear ODEs based upon the QR approximation theory for the Lyapunov spectrum and Theorems \ref{thm:maintheorem1} and \ref{thm:maintheorem2}.  As shown in Section \ref{sec:introduction} and Theorem \ref{thm:counterexample}, the AN-stability of BDF2 does not guarantee that there is no stability induced step-size restriction when solving time-dependent problems.  

We first show how to evaluate the underlying one-step method of a strictly stable GLM indirectly.  Theorem \ref{thm:maintheorem2} implies that the exponential stability of a strictly stable GLM solving \eqref{eq:2dlin} can be characterized by the Lyapunov or Sacker-Sell spectrum of an underlying one-step method 
\begin{equation}\label{eq:onestepnew}
y_{n+1} = H_1(y_n,n,h) \equiv \Phi^A(n;h)y_n
\end{equation}
Rather than attempting to directly evaluate the function $H_1$ we instead make use of \eqref{eq:maintheorem1est} to evaluate $H_1$ approximately.  Let $X_n:=X(n;X_0,t_0,h)$ denote the output of $\mathcal{M}$ applied to solve \eqref{eq:lininhomo2} using step-size $h > 0$, initial value $X_0 \in \mathbb{R}^{dk}$, and initial time $t_0 > s$ and express $X(n;X_0,t_0,h)=((x_n^1)^T,\hdots,(x_n^k)^T)^T$.  For the sequence defined by $Y_n = (P^{-1} \otimes I_d)X_n$ with $Y_n = ((y_n^1)^T,\hdots,(y_n^k)^T)^T$ there exists $G > 0$, $\gamma \in (0,1)$, and $Z_n$ of the form $Z_n = (z_n,\varphi(z_n,0,h))^T$, where $\varphi$ is as defined in Theorem \ref{thm:maintheorem1},  so that if $n \geq 0$, then 
\begin{equation}\label{eq:uniformdecay}
\|Y_n - Z_n \| \leq G \gamma^{n}
\end{equation}
and $z_{n+1} = H_1(z_n,n,h)$.  If we let $P^{-1} = (\overline{p}_{i,j})_{i,j=1}^{k}$, then if follows from \eqref{eq:uniformdecay} that the sequence defined component-wise as $w_n := \sum_{j=1}^{k}\overline{p}_{1,j} x_n^j$ is approximately equal to an output of \eqref{eq:onestepnew} for sufficiently large values of $n \geq 0$.  

We use this technique to approximate the largest discrete Lyapunov exponent of \eqref{eq:onestepnew} as follows.  Given an initial condition $x(0)=x_0$ we use the RK4 Runge-Kutta method to compute $x_1$.  For $n \geq 2$, we solve the BDF2 equation \eqref{eq:2dlin} for $x_{n+2}$ and set $X_n=(x_n^T,x_{n+1}^T)^T$.  Using $X_n$, we form $w_n = \sum_{j=1}^{3}\overline{p}_{1,j} x_{n+j-1}$.  Since $w_n$ approximately satisfies \eqref{eq:onestepnew} we can view it as the first column in a fundamental matrix solution.  Suppose that we let $w_n = Q_n R_n$ be a QR factorization where $Q_n \in \mathbb{R}^{d\times 1}$ is orthogonal and $R_n \in \mathbb{R}^{1\times 1}$.  Under the assumption that \eqref{eq:onestepnew} has a discrete integral separation structure, the largest discrete Lyapunov exponent $\mu_{\text{max}}$ of \eqref{eq:onestepnew} is almost surely (see \cite{DVV95} and also \cite{EckemannRuelle} and \cite{JPS1987}) given by
\begin{equation}\label{eq:lle}
\mu_{\text{max}} = \limsup_{n \rightarrow \infty}\dfrac{1}{t_n-t_0}\sum_{j=0}^{n} \ln((R_j)_{1,1})
\end{equation}
where $(R_n)_{1,1}$ denotes the $(1,1)$ entry of $R_n$.  We estimate \eqref{eq:lle} as
\begin{equation}\label{eq:lleappr}
\mu_{\text{appr}}(N_0,N) = \text{max}_{N_0 \leq n \leq N_0 + N}\dfrac{1}{t_n-t_0}\sum_{j=N_0}^{n} \ln((R_j)_{1,1}).
\end{equation}
We approximate the largest discrete Lyapunov exponent $\mu_{\text{max}}$ of \eqref{eq:onestepnew} by \eqref{eq:lleappr} using and use the sign of $\mu_{\text{appr}}(N_0,N)$ for large values of $N_0$ and $N$ as a stability diagnostic for the numerical solution of \eqref{eq:2dlin} by \eqref{eq:bdf2}.  Note that conclusion 2 of Theorem \ref{thm:onesteptheorem} and the fact that $x(0)=x_0 = (1,0)^T$ implies that almost surely we have
$$\mu_{\text{appr}}(N_0,N) = \frac{a_1}{h(N-N_0)}(\cos(Nh)-\cos(N_0 h)) + b_1 +\mathcal{O}(h^{2})$$
so that $\mu_{\text{appr}}(N_0,N) \approx b_1 +\mathcal{O}(h^2)$ as $N-N_0 \rightarrow \infty$ .

\begin{table}[h!]
    \begin{tabular}{ | l | l | l | l | l|}
    \hline
    $h$ & LTEmean & LTEmax & $\mu_{\text{appr}}(N_f/2,N_f/2)$   \\  \hline\hline
     $7.5E-1$ & $1.37E10$ & $1.51E11$ & $7.68E-1$  \\\hline
     $7.5E-2$ & $3.75E-3$  & $9.42E-3$  & $9.03E-3$    \\\hline
     $7.5E-3$ & $3.60E-7$ &  $6.38E-4$ & $-9.70E-2$     \\\hline
     $7.5E-4$ & $1.95E-9$  & $6.24E-5$  & $-9.04E-2$    \\\hline
    \end{tabular}
\caption{Results of an experiment for the solution of \protect\eqref{eq:2dlin} using BDF2, $a_1=a_2=1.2$, $b_1 = -0.14$, $b_2=-0.15$, $\beta=10.0$, $\omega = 1$, and a final time of $t_f = 40$ for various step-sizes $h$ and the initial condition $x(0)=(1,0)^T$.  LTEmean is the mean local truncation error, LTEmax is the maximum local truncation error, and $\ensuremath{\mu_{\text{appr}}(N_f/2,N_f/2)}$ is the value of \protect\eqref{eq:lleappr} where $N_f$ is the final step of the approximation. }
\label{table}
\end{table}

\begin{figure}
\hspace{-0.20in}
\hbox{\includegraphics[scale=0.44]{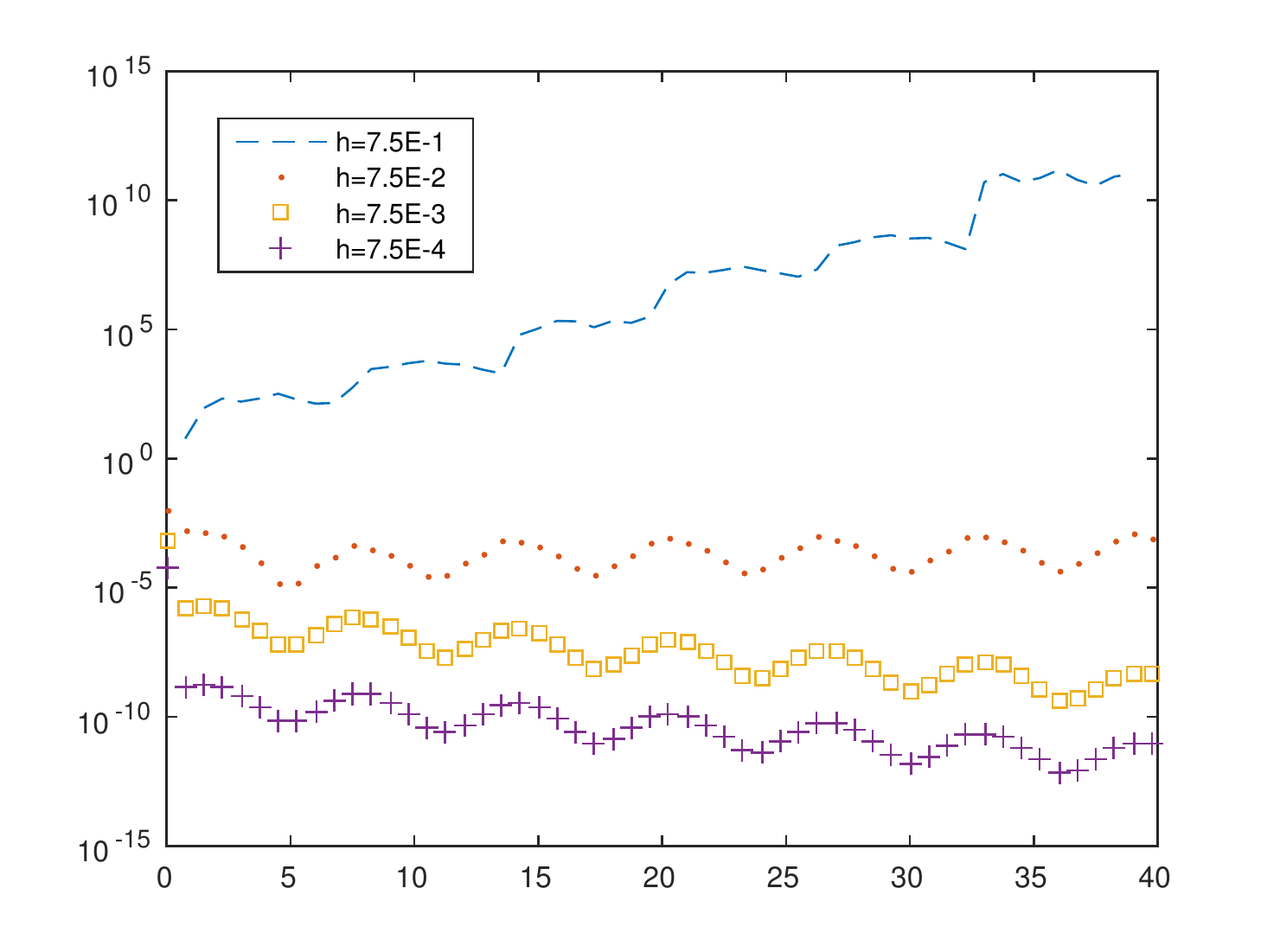}\includegraphics[scale=0.44]{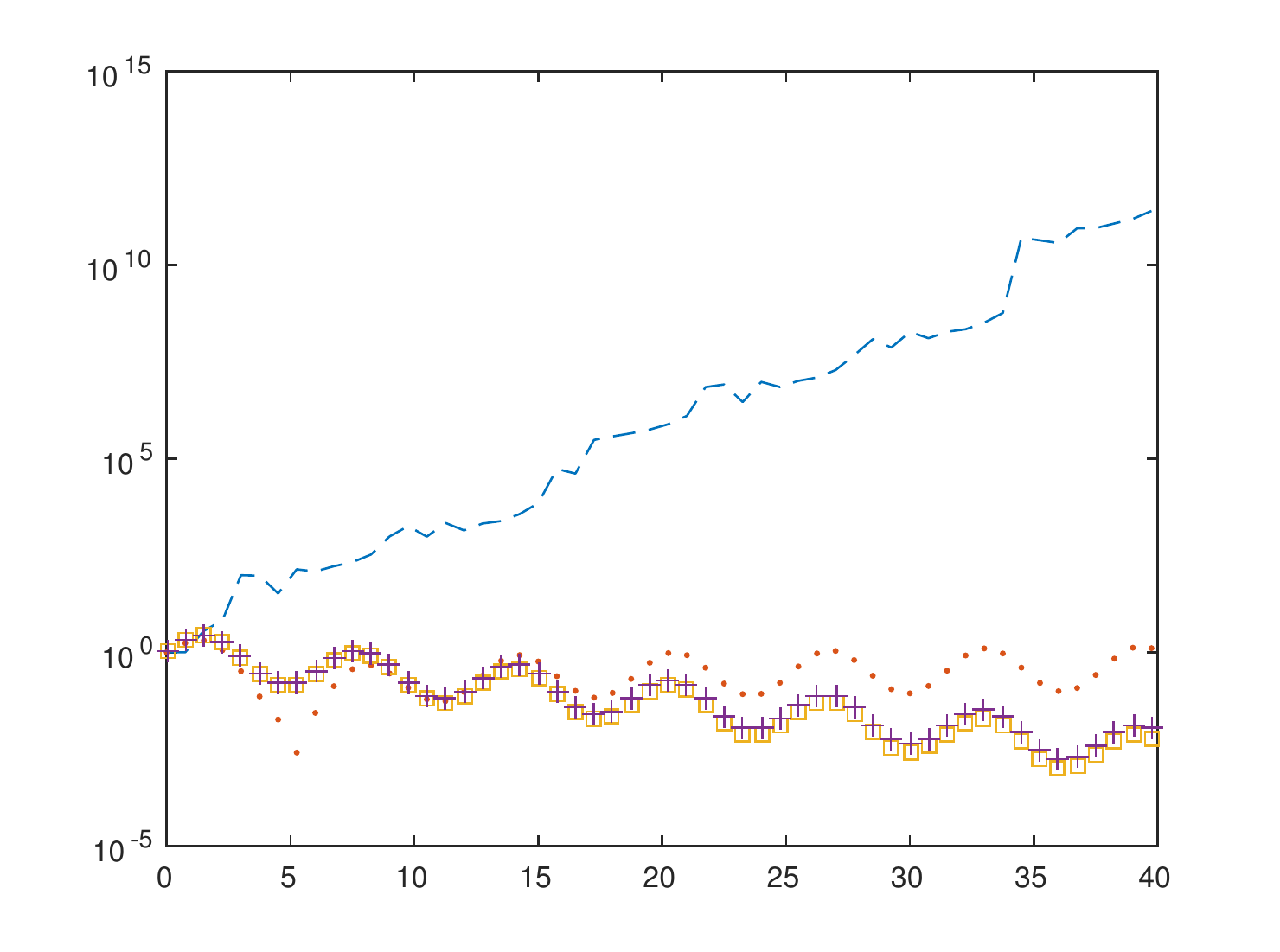}}
\caption{Left: Logarithmic plot of the 2-norm of the local truncation error of the numerical solution versus time for various values of $h$.  Right: Logarithmic plot of the 2-norm of the numerical solution versus time for various values of $h$.  The parameter values used were $a_1=a_2=1.2$, $b_1 = -0.14$, $b_2=-0.15$, $\beta=10.0$, $\omega = 1$ with a final time of $t_f = 40$ and the initial condition $x(0)=(1,0)^T$.  }
\label{figure}
\end{figure}

We display the results of our first experiment in Table \ref{table} and Figure \ref{figure}.  For step-sizes $h=7.5\cdot 10^{-1},7.5 \cdot 10^{-2}$ the method \eqref{eq:bdf2} produces numerical solutions to \eqref{eq:2dlin} that are growing in norm with approximate largest discrete Lyapunov exponents that are positive.  When $h=7.5 \cdot 10^{-2}$ the local trunation error, which is gradually increasing as shown in Figure \ref{figure}, remains bounded by $10^{-2}$.  When $h =7.5 \cdot 10^{-3}, 7.5 \cdot 10^{-4}$ the method \eqref{eq:bdf2} produces a decaying solution to \eqref{eq:2dlin} and the approximate largest discrete Lyapunov exponent of \eqref{eq:onestepnew} is negative.  This experiment shows that monitoring the approximate largest discrete Lyapunov exponent of the one-step method \eqref{eq:onestepnew} can be a more effective tool for controlling the global error and monitoring stability than the local truncation error.
\begin{table}[ht!]
%\hspace{.75in}
    \begin{tabular}{ | l | l | l | l | l| l|}
    \hline
    $a_1=a_2=a$ & LTEmean & LTEmax & $\mu_{\text{appr}}(N_f/2,N_f/2)$ & $\tau_{\text{max}}$    \\  \hline\hline
     $1.15$ & $5.50E-5$ & $4.38E-3$ & $-2.33E-2$ & $1.068$ \\\hline
     $1.45$ & $1.18E-4$  & $5.02E-3$  & $-1.69E-3$ & $1.086$    \\\hline
     $1.75$ & $2.88E-4$ &  $5.70E-3$ & $1.78E-2$   & $1.11$  \\\hline
     $2.05$ & $7.96E-4$  & $6.4E-3$  & $3.64E-2$  & $1.23$ \\\hline
    \end{tabular}
\caption{Results of an experiment for the solution of \protect\eqref{eq:2dlin} using BDF2, using $b_1 = -0.5$, $b_2=-.055$, $\beta=1.0$, $\omega = 1$, and a final time of $t_f = 100$ for various values of $a=a_1=a_2$ using the step-sizes $h=0.05$ and the initial condition $x(0)=(1,0)^T$.  LTEmean is the mean local truncation error, LTEmax is the maximum local truncation error, $\ensuremath{\mu_{\text{appr}}(N_f/2,N_f/2)}$ is the value of \protect\eqref{eq:lleappr} where $N_f$ is the final step of the approximation, and $\ensuremath{\tau_{\text{max}}}$ is the maximum value of $\tau_n$ which denotes the quotient of the local truncation error at time-steps $n+1$ and $n$.  }
\label{table2}
\end{table}

\begin{figure}
\hspace{-.2in}
\includegraphics[scale=0.44]{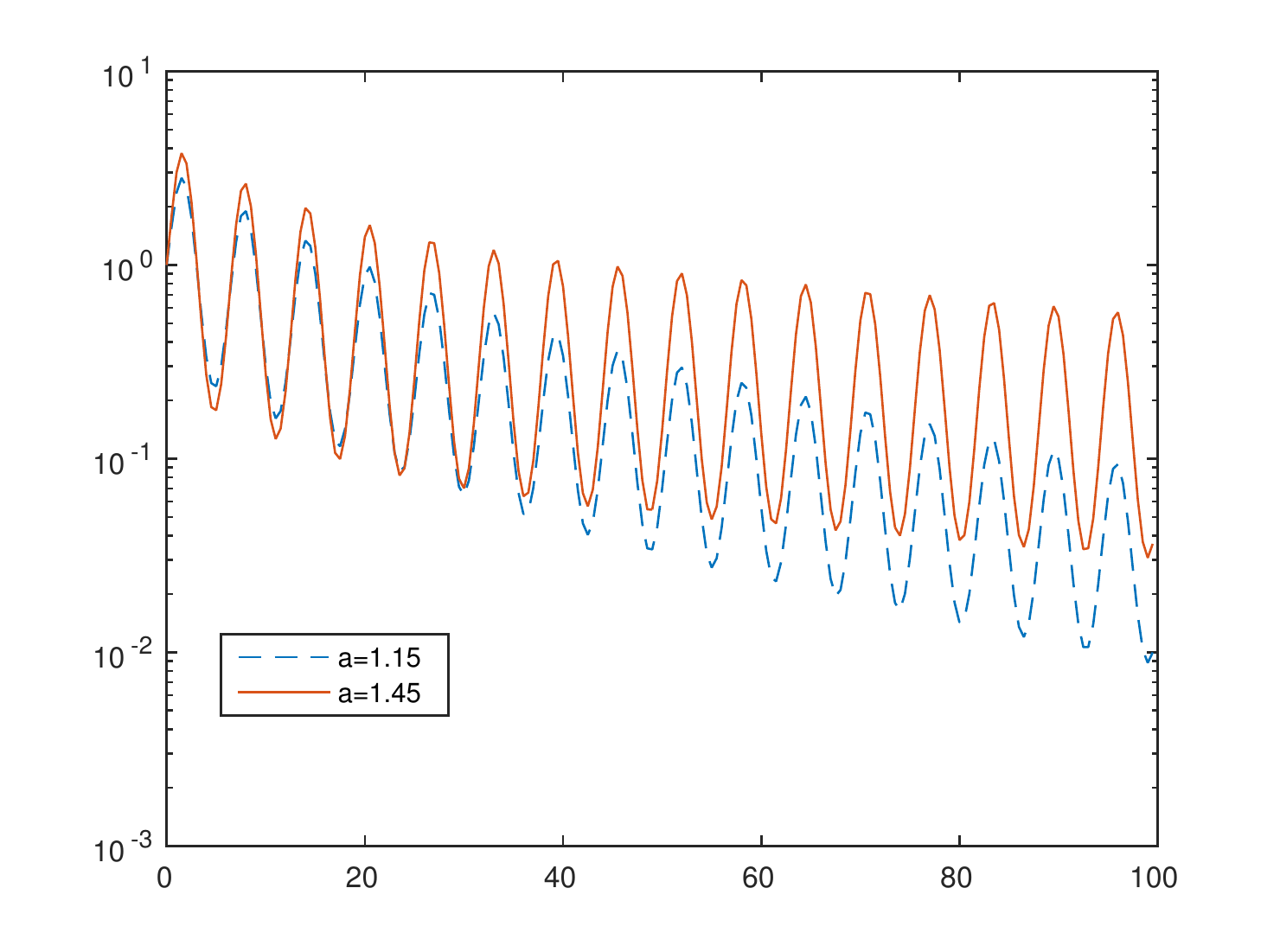}\includegraphics[scale=0.44]{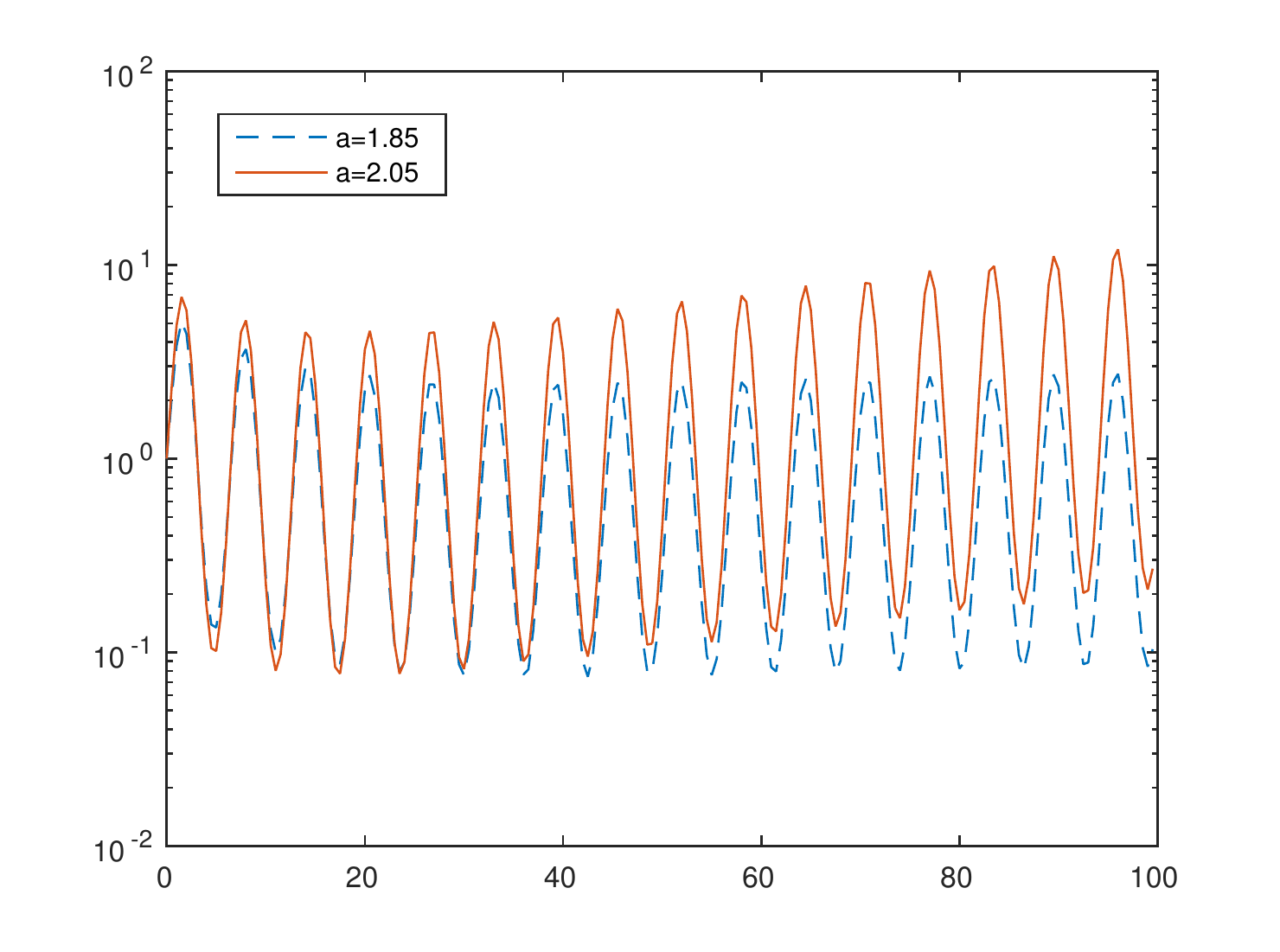}
\caption{Left: Logarithmic plot of the 2-norm of the local truncation error of the numerical solution versus time for various values of $h$.  Right: Logarithmic plot of the 2-norm of the numerical solution versus time for various values of $h$.  The parameter values used were using $b_1 = -0.5$, $b_2=-.055$, $\beta=1.0$, $\omega = 1$, and a final time of $t_f = 100$ for various values of $a=a_1=a_2$ using the step-sizes $h=0.05$ and the initial condition $x(0)=(1,0)^T$.  }
\label{figure2}
\end{figure}

In Table \ref{table2} and Figure \ref{figure2} we display the results of our second experiment.  The results of this experiment are meant to illustrate the difficulty in detecting stability using only point-wise values of the local truncation error.  We see that there are no spikes in the local truncation error from one step to the next since $\tau_{\text{max}}$ is approximately $1$ for all values of $a=a_1=a_2$.  Additionally, as the parameter $a$ varies from $1.45$ to $1.75$, the numerical solution becomes unstable and the ratio between the mean and maximum 2-norm of the local truncation error is $2.44$ and $1.14$ respectively which are comparable in value to the corresponding ratios when the parameter $a$ varies from $1.15$ to $1.45$ where there is no loss of stability.  This experiment demonstrates that the point-wise local truncation error and its local variation can fail to detect a loss of time-dependent stability.

\section{Conclusion}\label{sec:conclusion}
In this work we have used invariant manifold theory for nonautonomous difference equations to show that a strictly stable GLM solving a nonautonomous ODE that satisfies a global Lipschitz condition has an underlying one-step method whenever the step-size is sufficiently small.  This result combined with the Lyapunov and Sacker-Sell spectral stability theory for one-step methods developed in \cite{SVV2016,SVV1} and \cite{steyerthesis} is applied to analyze the stability of a strictly stable GLM solving a nonautonomous linear ODE.  These theoretical results are then applied to show that sign of the approximate largest discrete Lyapunov exponent of the underlying one-step method of a strictly stable GLM can be a more robust tool than the point-wise values of the local truncation error for monitoring the stability (and hence global error) of the numerical solution of a nonautonomous linear ODE IVP.  

Most step-size selection strategies for the solution of ODE IVPs select step-size based mainly on the local accuracy of the method, which we have shown in Section \ref{sec:experiments} can cause a solver to produce an exponentially growing approximation to an exponentially contracting nonautonomous linear ODE, even if the method is AN-stable.  Our experimental results suggest that the nonautonomous stability theory for GLMs that we have developed can be a useful tool for step-size selection based on stability as well as accuracy (a practical step-size selection algorithm for explicit Runge-Kutta methods based on these ideas can be found in \cite{SVV2016}).  In future work it remains to show that our results can be extended to variable step-size and variable order GLMs. Interestingly, whereas we have used our nonautonomous results as a practical way of detecting (and hence correcting) an unstable numerical solution, in the abstract of \cite{K1986} it is stated that "...this result is of theoretical interest; it does not seem to affect the significance of multi-step methods for practical computations".  The results of the present paper that build upon the fundamental ideas in \cite{K1986} and \cite{EN1988} serve as yet another example of how mathematics that is considered theoretical and abstract can potentially find a practical application.\\

\textbf{Acknowledgments} The authors would like to extend thanks to the referees for their helpful comments.

% You may incorporate your references as follows in your main tex file.
% Using BibTex is not recommended but can be handled.

\medskip
% The data information below will be filled by AIMS editorial staff
Received xxxx 20xx; revised xxxx 20xx.
\medskip

\end{document}